\documentclass[preprint,10pt]{elsarticle}
\usepackage{amssymb}
\usepackage{pifont}
\usepackage{amsfonts}
\usepackage{amsmath}
\usepackage{graphicx}
\usepackage{latexsym,amsthm}
\usepackage{hyperref}
\usepackage{combelow}
\usepackage{setspace}
\usepackage[left=1.3cm,top=1.75cm,right=1.3cm]{geometry}

\makeatletter
\def\ps@pprintTitle{%
  \let\@oddhead\@empty
  \let\@evenhead\@empty
  \def\@oddfoot{\reset@font\hfil\thepage\hfil}
  \let\@evenfoot\@oddfoot
}
\makeatother
\allowdisplaybreaks

\newtheorem{lemma}{Lemma}

\newtheorem{theorem}{Theorem}
\numberwithin{equation}{section}

\begin{document}

\begin{frontmatter}
\title{On King Type Modification of $(p,q)$ Baskakov Operators which preserves $x^2$}
\author[label1, label*]{Vishnu Narayan Mishra}
\ead{vishnunarayanmishra@gmail.com,
vishnu\_narayanmishra@yahoo.co.in}
\fntext[label*]{Corresponding author}
\author[label1]{Shikha Pandey}
\ead{sp1486@gmail.com}
\author[label2,label3]{Lakshmi Narayan Mishra}
\ead{lakshminarayanmishra04@gmail.com}
\address[label1]{Department of Applied Mathematics \& Humanities,
Sardar Vallabhbhai National Institute of Technology, Ichchhanath Mahadev Dumas Road, Surat -395 007 (Gujarat), India}
\address[label2] {Department of Mathematics, National Institute of Technology, Silchar-788010, Assam, India.}
\address[label3]{L. 1627 Awadh Puri Colony Phase -III, Beniganj, Opp. Industrial Training Institute (I.T.I.), Ayodhya Main Road, Faizabad-224 001, \\Uttar Pradesh, India } 

\begin{abstract}
In the present paper, we construct and investigate a variant of modified $(p,q)$-Baskakov operators, which reproduce the
test function $x^2$. The order of approximation of the operators via K-functional and second order, usual modulus of continuity, weighted approximation properties are disscussed. In the end some graphical results, comparison with $(p,q)$-Baskakov operators is explained.
\end{abstract}
\end{frontmatter}
\section{Introduction}
The theory of approximation is a very extensive field
and the study of the approximation via $q-$calculus and $(p,q)-$calculus is 
of great mathematical interest and of great practical importance.
Positive approximation processes play an important role in Approximation
Theory and appear in a very natural way dealing with
approximation of continuous functions, especially one, which requires
further qualitative properties such as monotonicity, convexity
and shape preservation and so on. In recent years lot of  research is being done on the positive linear operators based on 
$q$- calculus (See \cite{lnm1},\cite{HM1},\cite{HM2},\cite{HM3}) and $(p,q)-$calculus (See \cite{acar2},\cite{ acar3}, \cite{bernpq},\cite{bbhpq}).\\ 
Basics of $(p,q)-$calculus is:\\
Assume $0<q<p\leq 1$, for any non-negative integer $n$,the $(p,q)-$integer $[n]_{p,q}$, $(p,q)-$factorial $[n]_{p,q}!$, $(p,q)-$ binomial coefficient, $(p,q)-$ power expansion, $(p,q)-$derivative is defined as
$$[n]_{p,q} := \dfrac{p^n-q^n}{p-q},$$
$$[n]_{p,q}!:=[n]_{p,q} [n-1]_{p,q} [n-2]_{p,q} \ldots 1,$$
$$\left[\begin{array}{c} n \\ k \end{array} \right]_{p,q} := \dfrac{[n]_{p,q}!}{[k]_{p,q}![n-k]_{p,q}!},$$
$$(ax+by)_{p,q}^n:=\sum_{k=0}^n\left[\begin{array}{c} n \\ k \end{array} \right]_{p,q}a^{n-k}b^k x^{n-k}y^k,$$
$$(x+y)_{p,q}^n := (x+y) (px+qy) (p^2x+q^2y)\ldots (p^{n-1}x+q^{n-1}y)=\prod_{j=0}^{n-1}(p^jx+q^jy)$$
$$(D_{p,q}f)(x):=\frac{f(px)-f(qx)}{(p-q)x}, x \neq 0, ~~ (D_{p,q}f)(0):= f^{'}(0)$$ provided that $f$ is differentiable at $0$.\\
$(p,q)-$ derivative of product of two functions $u(x)$ and $v(x)$ is
$$D_{p,q}(u(x)v(x)):=D_{p,q}(u(x))v(qx)+u(qx)D_{p,q}(v(x)).$$
For more details about $(p,q)-$ calculus readers may refer (\cite{basicpq1}-\cite{basicpq3},\cite{basicpq4}).
\section{Construction}
Baskakov operators and their $q$-analogue as in (\cite{aral1},\cite{vnm1}) has been studied in order to approximate functions over unbounded intervals.
The $(p,q)$-analogue of Baskakov operators for $x \in [0,\infty)$ and $0 < q < p \leq 1$ is defined as in (\cite{acar1}-\cite{aral2})
\begin{equation}\label{operator1}
B_{n,p,q}(f;x)=\sum_{k=0}^n \left[\begin{array}{c} n+k-1 \\ k \end{array} \right]_{p,q} p^{k+n(n-1)/2}q^{k(k-1)/2}\frac{x^k}{(1\oplus x)^{n+k}_{p,q}} f\left(\frac{p^{n-1}[k]_{p,q}}{q^{k-1}[n]_{p,q}}\right)
\end{equation}
and its moments are
\begin{lemma}\label{lem1}
 For $0<q<p\leq 1$ and $n\in \mathbb{N}$. We have
\begin{equation*}
B_{n,p,q}(1;x)=1,~B_{n,p,q}(t;x)=x,~B_{n,p,q}(t^2;x)=x^2+\frac{p^{n-1}x}{[n]_{p,q}}\left(1+\frac{p}{q}x\right)
\end{equation*}
\end{lemma}
Motivated by it, we propose the King type modification of $(p,q)$-analogue of Baskakov operators which preserves $x^2$ as
\begin{equation}\label{operator2}
B_{n,p,q}^*(f;x)=\sum_{k=0}^n \left[\begin{array}{c} n+k-1 \\ k \end{array} \right]_{p,q} p^{k+n(n-1)/2}q^{k(k-1)/2}\frac{r_n^k(x)}{(1\oplus r_n(x))^{n+k}_{p,q}} f\left(\frac{p^{n-1}[k]_{p,q}}{q^{k-1}[n]_{p,q}}\right)
\end{equation}
where
\begin{equation*}
r_n(x)=\frac{-p^{n-1}+\sqrt{p^{2n-2}+4[n]_{p,q}([n]_{p,q}+p^n/q)x^2}}{2([n]_{p,q}+p^n/q)},\qquad r_n(x)\geq 0, ~ x\in [0,\infty)
\end{equation*}
\begin{lemma}\label{lem2}
For $0<q<p\leq 1$ and $n\in \mathbb{N}$. We have
\begin{equation}\label{1EqLem2}
B_{n,p,q}^*(1;x)=1,B_{n,p,q}^*(t;x)=r_n(x), B_{n,p,q}^*(t^2;x)= x^2.
\end{equation}
\end{lemma}
\begin{lemma}\label{lem3}
For $0<q<p\leq 1$ and $n\in \mathbb{N}$. We have
\begin{equation}\label{1EqLem3}
B_{n,p,q}^*(t-x;x)=r_n(x)-x,
\end{equation}
\begin{equation}\label{2EqLem3}
B_{n,p,q}^*\left((t-x)^2;x\right)=2x^2-2xr_n(x),
\end{equation}
and $$B_{n,p,q}^*(t-x;x) \leq \left(\frac{\sqrt{\frac{p^n}{q}}-\frac{p^n}{q\sqrt{[n]_{p,q}}}}{\sqrt{[n]_{p,q}}+\frac{p^n}{q\sqrt{[n]_{p,q}}}}\right)x ,$$
\begin{equation}\label{3EqLem3}
B_{n,p,q}^*\left((t-x)^2;x\right)\leq \frac{2\left(\sqrt{\frac{p^n}{q}}-\frac{p^n}{q\sqrt{[n]_{p,q}}}\right)x^2}{\sqrt{[n]_{p,q}}+\frac{p^n}{q\sqrt{[n]_{p,q}}}}+\frac{p^{n-1}x}{\frac{p^n}{q}+[n]_{p,q}}.
\end{equation}
\end{lemma}
\begin{proof}
The first two equalities are clear from Lemma 2. For the inequalities, using \ref{operator2},\ref{1EqLem3} and \ref{2EqLem3}. We can write
\begin{eqnarray*}
|B_{n,p,q}^*(t-x;x)|&=&\left|\frac{-p^{n-1}+\sqrt{p^{2n-2}+4[n]_{p,q}([n]_{p,q}+p^n/q)x^2}}{2([n]_{p,q}+p^n/q)}-x\right|
\\&\leq & \left(\frac{\sqrt{\frac{p^n}{q}}-\frac{p^n}{q\sqrt{[n]_{p,q}}}}{\sqrt{[n]_{p,q}}+\frac{p^n}{q\sqrt{[n]_{p,q}}}}\right)x ,
\end{eqnarray*}
and
\begin{eqnarray*}
B_{n,p,q}^*\left((t-x)^2;x\right)&=&2x^2-x\left(\frac{-p^{n-1}+\sqrt{p^{2n-2}+4[n]_{p,q}([n]_{p,q}+p^n/q)x^2}}{([n]_{p,q}+p^n/q)}\right)
\\&=& \frac{2x^2([n]_{p,q}+p^n/q)+xp^{n-1}-x\sqrt{p^{2n-2}+4[n]_{p,q}([n]_{p,q}+p^n/q)x^2}}{([n]_{p,q}+p^n/q)}
\\
&\leq & \frac{2\left(\sqrt{\frac{p^n}{q}}-\frac{p^n}{q\sqrt{[n]_{p,q}}}\right)x^2}{\sqrt{[n]_{p,q}}+\frac{p^n}{q\sqrt{[n]_{p,q}}}}+\frac{p^{n-1}x}{\frac{p^n}{q}+[n]_{p,q}}.
\end{eqnarray*}
\end{proof}

\section{Approximation Properties of $B_{n,p,q}^*f$}
Now we will study the weighted approximation theorem for the operator \ref{operator2}. Considering the following defintions
\begin{enumerate}
\item[$\cdot$] $C_B[0,\infty)$ is the space of real-valued uniformly continuous and bounded
functions $f$ defined on the interval $[0,\infty)$.
\item[$\cdot$]Norm $\|\cdot\|$ on the space $C_B[0,\infty)$ is given by $$\|f\|=\sup_{o\leq x <\infty}|f(x)|.$$
\item[$\cdot$] $\mathcal{K}-$ functional is
$$\mathcal{K}_2(f,\delta)=\inf_{g\ \mathcal{W}^2}\{\|f-g\|+\delta\|g^{''}\|\},$$
where $\delta >0$ and $\mathcal{W}^2=\{g\in C_B[0,\infty): g,g^{''}\in C_B[0,\infty) \}$. By [\cite{devore}, p-177, Theorem 2.4] we have an absolute constant $\mathcal{C}>0$ such that 
\begin{equation}\label{Ineq1}
\mathcal{K}_2(f,\delta)\leq \mathcal{C} \omega_2(f,\sqrt{\delta}).
\end{equation}
\item[$\cdot$] The second order modulus of smoothness of $f\in C_B[0,\infty)$ is 
$$\omega_2(f,\delta)=\sup_{0\leq h \leq \sqrt{\delta}} \sup_{0\leq x < \infty}|f(x+2h)-2f(x+h)+f(x)|,$$
and the usual modulus of continuity is given by
$$\omega (f,\delta)=\sup_{0\leq h \leq \delta} \sup_{0\leq x < \infty}|f(x+h)-f(x)|.$$
\item[$\cdot$] $B_m[0,\infty)$ is the space of all functions satisfying $|f(x)|\leq M_f (1+x^m), x\in [0,\infty), m>0, M_f$ is a constant depending on $f$, and $C_m[0,\infty)=B_m[0,\infty) \cap C[0,\infty)$.
\item[$\cdot$] The norm for $f \in C_m[0,\infty)$ is
$$\|f\|_m=\sup_{x\geq 0}\frac{|f(x)|}{1+x^m}.$$
\item[ $\cdot$] $$C_m^*[0,\infty)=\left\{ f\in C_m[0,\infty): \frac{|f(x)|}{1+x^m} < \infty \right\}.$$
\end{enumerate}
\begin{theorem}
Let $q=q_n\in (0,1), p=p_n \in (0,1]$ such that $q_n \to 1, ~p_n\to 1$ as $n\to \infty $. Then for each function $f\in C_2^*[0,\infty)$ we get $$\lim_{n\to \infty}\|B_{n,p_n,q_n}^*f-f\|_2=0.$$
\end{theorem}
\begin{proof}
As per weighted Korovkin theorem, it is sufficient to verify the following 3 conditions:
\begin{equation}\label{1EqThm1}
\lim_{n\to \infty}\|B_{n,p_n,q_n}^*e_i-e_i\|_2=0,\qquad i=0,1,2.
\end{equation}
From \ref{1EqLem2}, \ref{1EqThm1} hold for $i=0.$  For $i=1,2$, using \ref{1EqLem3}-\ref{3EqLem3} we have, as $n\to \infty$
\begin{eqnarray*}
\|B_{n,p_n,q_n}^*e_1-e_1\|_2 &=& \sup_{x\geq 0}\frac{r_n(x)-x}{1+x^2}\\
&\leq & \sup_{x\geq 0} \frac{x}{1+x^2} \left(\frac{\sqrt{\frac{p_n^n}{q_n}}-\frac{p_n^n}{q_n\sqrt{[n]_{p_n,q_n}}}}{\sqrt{[n]_{p_n,q_n}}+\frac{p_n^n}{q_n\sqrt{[n]_{p_n,q_n}}}}\right)\\
&\leq &  \left(\frac{\sqrt{\frac{p_n^n}{q_n}}-\frac{p_n^n}{q_n\sqrt{[n]_{p_n,q_n}}}}{\sqrt{[n]_{p_n,q_n}}+\frac{p_n^n}{q_n\sqrt{[n]_{p_n,q_n}}}}\right) \longrightarrow 0 
\end{eqnarray*}
and
\begin{eqnarray*}
\|B_{n,p_n,q_n}^*e_2-e_2\|_2 &\leq & \sup_{x\geq 0}\frac{\frac{2\left(\sqrt{\frac{p_n^n}{q_n}}-\frac{p_n^n}{q_n\sqrt{[n]_{p_n,q_n}}}\right)x^2}{\sqrt{[n]_{p_n,q_n}}+\frac{p_n^n}{q_n\sqrt{[n]_{p_n,q_n}}}}+\frac{p_n^{n-1}x}{\frac{p_n^n}{q_n}+[n]_{p_n,q_n}}}{1+x^2}\\
&\leq & \frac{2\left(\sqrt{\frac{p_n^n}{q_n}}-\frac{p_n^n}{q_n\sqrt{[n]_{p_n,q_n}}}\right)}{\sqrt{[n]_{p_n,q_n}}+\frac{p_n^n}{q_n\sqrt{[n]_{p_n,q_n}}}}+\frac{p_n^{n-1}}{\frac{p_n^n}{q_n}+[n]_{p_n,q_n}} \longrightarrow 0
\end{eqnarray*}
\end{proof}
\begin{theorem}
Let $p,q \in (0,1)$ such that $0<q<p \leq 1.$ Then we have 
$$|B_{n,p,q}^*(f;x)|\leq M \omega_2(f, \sqrt{\delta_n(x)})+\omega \left(f,\frac{\sqrt{\frac{p^n}{q}}-\frac{p^n}{q\sqrt{[n]_{p,q}}}}{\sqrt{[n]_{p,q}}+\frac{p^n}{q\sqrt{[n]_{p,q}}}}x\right),$$
where $x\in [0,\infty)$, $f\in C_B[0,\infty)$ and $\delta_n(x)=\frac{3\left(\sqrt{\frac{p^n}{q}}-\frac{p^n}{q\sqrt{[n]_{p,q}}}\right)}{\sqrt{[n]_{p,q}}+\frac{p^n}{q\sqrt{[n]_{p,q}}}}+\frac{p^{n-1}x}{\frac{p^n}{q}+[n]_{p,q}}.$
\end{theorem}
\begin{proof}
Consider the  auxiliary operator
\begin{equation} \label{operator3}
 \hat{B}_{n,p,q}^*(f;x)=B_{n,p,q}^*(f;x)+f(x)-f(r_n(x)).
 \end{equation}
 Hence it is clear by \ref{1EqLem2} that the operator \ref{operator3} are linear and  $$\hat{B}_{n,p,q}^*(1;x)={B}_{n,p,q}^*(1;x)=1,$$$$\hat{B}_{n,p,q}^*(t;x)=B_{n,p,q}^*(t;x)+x-r_n(x)=x.$$
 Let $g\in \mathcal{W}^2$, By classical Taylor's expansion of $g\in \mathcal{W}^2$ gives for $t \in \mathbb{R}^{+}$ that
 $$g(t)=g(x)+(t-x) g^{'}(x)+\int_x^t (t-u)g{''}(u) du.$$
 Applying with the operator \ref{operator3} on both sides  we get,
 $$\hat{B}_{n,p,q}^*(g(t);x)-g(x)=\hat{B}_{n,p,q}^*\left(\int_x^t (t-u)g{''}(u) du;x\right).$$
 By definition of \ref{operator3} we have
 \begin{eqnarray*}
 |\hat{B}_{n,p,q}^*(g;x)-g(x)|&\leq & \left|{B}_{n,p,q}^*\left(\int_x^t (t-u)g{''}(u) du;x\right) \right|+{B}_{n,p,q}^*\left(\int_x^{r_n(x)} (r_n(x)-u)g{''}(u) du;x\right)\\
 & \leq & \|g^{''}\|\left[ {B}_{n,p,q}^*\left((t-x)^2;x\right)+(r_n(x)-x)^2\right]\\
 &\leq & \|g^{''}\| \left[\frac{2\left(\sqrt{\frac{p^n}{q}}-\frac{p^n}{q\sqrt{[n]_{p,q}}}\right)x^2}{\sqrt{[n]_{p,q}}+\frac{p^n}{q\sqrt{[n]_{p,q}}}}+\frac{p^{n-1}x}{\frac{p^n}{q}+[n]_{p,q}}+(r_n(x)-x)^2 \right]\\
 &\leq &  \|g^{''}\|\left[\frac{2\left(\sqrt{\frac{p^n}{q}}-\frac{p^n}{q\sqrt{[n]_{p,q}}}\right)x^2}{\sqrt{[n]_{p,q}}+\frac{p^n}{q\sqrt{[n]_{p,q}}}}+\frac{p^{n-1}x}{\frac{p^n}{q}+[n]_{p,q}}+\left(\frac{\sqrt{\frac{p^n}{q}}-\frac{p^n}{q\sqrt{[n]_{p,q}}}}{\sqrt{[n]_{p,q}}+\frac{p^n}{q\sqrt{[n]_{p,q}}}}\right)^2x^2 \right]\\
 &\leq &  \|g^{''}\|\left[\frac{3\left(\sqrt{\frac{p^n}{q}}-\frac{p^n}{q\sqrt{[n]_{p,q}}}\right)x^2}{\sqrt{[n]_{p,q}}+\frac{p^n}{q\sqrt{[n]_{p,q}}}}+\frac{p^{n-1}x}{\frac{p^n}{q}+[n]_{p,q}}\right]
\end{eqnarray*}
Using \ref{operator3}, \ref{operator2} and lemma \ref{lem2} we get
$$\left|\hat{B}_{n,p,q}^*(f;x)\right|\leq \left|{B}_{n,p,q}^*(f;x)\right|+2\|f\|\leq \|f\|{B}_{n,p,q}^*(1;x)+2\|f\| \leq 3 \|f\|.$$
We can obtain
\begin{eqnarray*}
\left|{B}_{n,p,q}^*(f;x)-f(x)\right| &\leq & \left|\hat{B}_{n,p,q}^*(f-g;x)-(f-g)(x)\right|+\left|\hat{B}_{n,p,q}^*(g;x)-g(x)\right|+|f(x)-f(r_n(x))|\\&\leq & 4 \|f-g\|+\|g^{''}\|\left(\frac{3\left(\sqrt{\frac{p^n}{q}}-\frac{p^n}{q\sqrt{[n]_{p,q}}}\right)x^2}{\sqrt{[n]_{p,q}}+\frac{p^n}{q\sqrt{[n]_{p,q}}}}+\frac{p^{n-1}x}{\frac{p^n}{q}+[n]_{p,q}} \right) +\omega (f,\frac{\sqrt{\frac{p^n}{q}}-\frac{p^n}{q\sqrt{[n]_{p,q}}}}{\sqrt{[n]_{p,q}}+\frac{p^n}{q\sqrt{[n]_{p,q}}}}x)
\end{eqnarray*}
 Upon taking infimum on right hand side over all $g\in \mathcal{W}^2$ we have
 $$|{B}_{n,p,q}^*(f;x)-f(x)| \leq 4\mathcal{K}_2(f,\delta)+\omega\left(f,\frac{\sqrt{\frac{p^n}{q}}-\frac{p^n}{q\sqrt{[n]_{p,q}}}}{\sqrt{[n]_{p,q}}+\frac{p^n}{q\sqrt{[n]_{p,q}}}}x\right).$$ 
 Using inequality \ref{Ineq1} we have
 $$|{B}_{n,p,q}^*(f;x)-f(x)| \leq M \omega_2(f,\sqrt{\delta)}+\omega\left(f,\frac{\sqrt{\frac{p^n}{q}}-\frac{p^n}{q\sqrt{[n]_{p,q}}}}{\sqrt{[n]_{p,q}}+\frac{p^n}{q\sqrt{[n]_{p,q}}}}x\right).$$ 
\end{proof}
\begin{theorem}
Let $q=q_n\in (0,1), p=p_n\in (0,1]$ such that $q_n\to 1,p_n\to 1$ as $n\to \infty$. If $f\in C_m^*[0,\infty)$ and let $f^*(z)=f(z^2), z\in [0,\infty).$ For all $t>0$ and $x\geq 0$, we have
 $$|{B}_{n,p_n,q_n}^*(f;x)-f(x)| \leq 2\omega \left(f^*,\sqrt{\frac{2\left(\sqrt{\frac{p^n}{q}}-\frac{p^n}{q\sqrt{[n]_{p,q}}}\right)x}{\sqrt{[n]_{p,q}}+\frac{p^n}{q\sqrt{[n]_{p,q}}}}+\frac{p^{n-1}}{\frac{p^n}{q}+[n]_{p,q}}}\right)$$
\end{theorem}
\begin{proof}
Let $f\in C_m^*[0,\infty)$ is fixed. By definition of $f^*$ we get
$${B}_{n,p_n,q_n}^*(f;x)={B}_{n,p_n,q_n}^*(f^*(\sqrt{\cdot});x).$$
Now, we can write 
\begin{eqnarray*}
|{B}_{n,p_n,q_n}^*(f;x)-f(x)|&=& |{B}_{n,p_n,q_n}^*(f^*(\sqrt{\cdot});x)-f^*(\sqrt{x})|\\
&=& \left|\sum_{k=0}^\infty\left( f^*\left(\sqrt{\frac{p_n^{n-1}[k]_{p_n,q_n}}{q_n^{k-1}[n]_{p_n,q_n}}}\right)-f^*(\sqrt{x})\right)b_{n,k}(p_n,q_n;r_n(x)) \right|\\
&\leq & \sum_{k=0}^\infty \left|f^*\left(\sqrt{\frac{p_n^{n-1}[k]_{p_n,q_n}}{q_n^{k-1}[n]_{p_n,q_n}}}\right)-f^*(\sqrt{x})\right| b_{n,k}(p_n,q_n;r_n(x))\\
&\leq & \sum_{k=0}^\infty \omega\left(f^*;\left| \left(\sqrt{\frac{p_n^{n-1}[k]_{p_n,q_n}}{q_n^{k-1}[n]_{p_n,q_n}}}\right)-\sqrt{x}\right| \right)b_{n,k}(p_n,q_n;r_n(x))
\\&=& \sum_{k=0}^\infty \omega\left(f^*;\frac{\left| \left(\sqrt{\frac{p_n^{n-1}[k]_{p_n,q_n}}{q_n^{k-1}[n]_{p_n,q_n}}}\right)-\sqrt{x}\right|}{{B}_{n,p_n,q_n}(|\sqrt{\cdot}-\sqrt{x}|;x)}{B}_{n,p_n,q_n}(|\sqrt{\cdot}-\sqrt{x}|;x) \right)b_{n,k}(p_n,q_n;r_n(x))
\end{eqnarray*}
where 
$$b_{n,k}(p_n,q_n;r_n(x))=\left[\begin{array}{c} n+k-1 \\ k \end{array} \right]_{p_n,q_n} p_n^{k+n(n-1)/2}q_n^{k(k-1)/2}\frac{r_n^k(x)}{(1\oplus r_n(x))^{n+k}_{p_n,q_n}}.$$
Using the property of modulus of continuity
$$\omega(f^*;\alpha\delta)\leq (1+\alpha) ~\omega(f^*;\delta), ~\alpha, \delta \geq 0,$$
we obtain
\begin{eqnarray*}
&& |{B}_{n,p,q}^*(f;x)-f(x)| \\&&
\leq \omega (f^*;B_{n,p_n,q_n}^*(|\sqrt{\cdot}-\sqrt{x}|;x))\sum_{k=0}^\infty \left(1+\frac{\left| \left(\sqrt{\frac{p_n^{n-1}[k]_{p_n,q_n}}{q_n^{k-1}[n]_{p_n,q_n}}}\right)-\sqrt{x}\right|}{{B}_{n,p_n,q_n}(|\sqrt{\cdot}-\sqrt{x}|;x)}\right) b_{n,k}(p_n,q_n;r_n(x))
\\&&=2\omega (f^*;B_{n,p_n,q_n}^*(|\sqrt{\cdot}-\sqrt{x}|;x))
\end{eqnarray*}
Since $\frac{1}{\left| \left(\sqrt{\frac{p_n^{n-1}[k]_{p_n,q_n}}{q_n^{k-1}[n]_{p_n,q_n}}}\right)+\sqrt{x}\right|}\leq \frac{1}{\sqrt{x}}$ and using Cauchy-Schwarz inequality, we get
\begin{eqnarray*}
{B}_{n,p_n,q_n}^*(|\sqrt{\cdot}-\sqrt{x}|;x)&=& \sum_{k=0}^\infty \left| \left(\sqrt{\frac{p_n^{n-1}[k]_{p_n,q_n}}{q_n^{k-1}[n]_{p_n,q_n}}}\right)-\sqrt{x}\right|b_{n,k}(p_n,q_n;r_n(x))
\\&=& \sum_{k=0}^\infty \frac{\left|\frac{p_n^{n-1}[k]_{p_n,q_n}}{q_n^{k-1}[n]_{p_n,q_n}}-x\right|}{\left(\sqrt{\frac{p_n^{n-1}[k]_{p_n,q_n}}{q_n^{k-1}[n]_{p_n,q_n}}}\right)+\sqrt{x}}b_{n,k}(p_n,q_n;r_n(x))\\
&\leq & \frac{1}{\sqrt{x}}\sum_{k=0}^\infty \left|\frac{p_n^{n-1}[k]_{p_n,q_n}}{q_n^{k-1}[n]_{p_n,q_n}}-x\right|b_{n,k}(p_n,q_n;r_n(x))\\
&\leq & \frac{1}{\sqrt{x}}\sqrt{\sum_{k=0}^\infty \left|\frac{p_n^{n-1}[k]_{p_n,q_n}}{q_n^{k-1}[n]_{p_n,q_n}}-x\right|^2b_{n,k}(p_n,q_n;r_n(x))}\\
&\leq & \frac{1}{\sqrt{x}}\sqrt{{B}_{n,p_n,q_n}^*((\cdot-x)^2;r_n(x))}\\
&\leq & \sqrt{\frac{2\left(\sqrt{\frac{p^n}{q}}-\frac{p^n}{q\sqrt{[n]_{p,q}}}\right)x}{\sqrt{[n]_{p,q}}+\frac{p^n}{q\sqrt{[n]_{p,q}}}}+\frac{p^{n-1}}{\frac{p^n}{q}+[n]_{p,q}}}
\end{eqnarray*}
which completes the proof.
\end{proof}
\section{Numerical Results}
Consider the function $f(x)=\sin x^2$. We have plotted the results for this function by both operators ${B}_{n,p,q}(\cdot;x)$ and ${B}_{n,p,q}^*(\cdot;x)$ for $n=2, p=0.9, q=0.8$. It is clear from the figure that King type $(p, q)$-Baskakov operator gives better approximate to the curve.
\begin{figure}[h!]
\begin{center}
\includegraphics[scale=0.7]{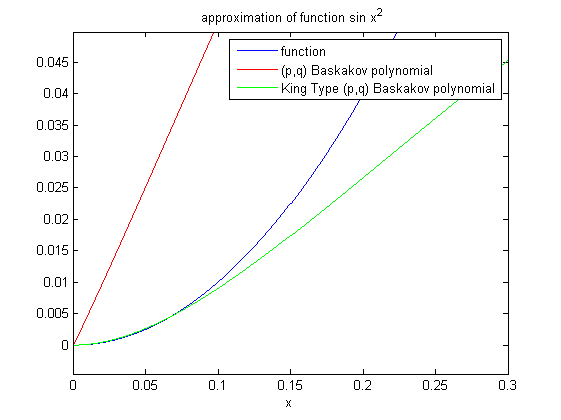}
\end{center}
\end{figure}

\begin{flushleft}
\textbf{Conflict of Interest} The authors declare that there is no conflict of interests regarding publication of this manuscript.
\end{flushleft}
\begin{flushleft}
\textbf{References}
\end{flushleft}

\end{document}